\documentclass[11pt, reqno]{amsart}
\usepackage[T1]{fontenc}
\usepackage{amsfonts}
\usepackage{amsmath}
\usepackage{amsthm}
\usepackage{amssymb}
\usepackage{geometry}
\usepackage{graphicx}
\usepackage{xcolor}
\usepackage{mathtools}
\usepackage[colorlinks=true, linkcolor=blue]{hyperref}
\usepackage{cleveref}
\usepackage[english]{babel}
\usepackage{lineno}
\usepackage{float}

\newtheorem{theorem}{Theorem}[section]

\newtheorem{lemma}[theorem]{Lemma}
\newtheorem{conjecture}[theorem]{Conjecture}
\newtheorem{proposition}[theorem]{Proposition}
\newtheorem{problem}[theorem]{Problem}
\usepackage{tikz}
\usetikzlibrary{positioning}
\usetikzlibrary{decorations,arrows}
\usetikzlibrary{decorations.markings}
\numberwithin{equation}{section}

\newcommand{\N}{\mathbb N}
\newcommand{\Z}{\mathbb Z}
\newcommand{\R}{\mathbb R}
\newcommand{\Q}{\mathbb Q}

\usepackage{rustic}
\usepackage[T1]{fontenc}

\emergencystretch=\maxdimen
\title{Threshold numbers of some graphs}
\author{Runze Wang}
\address[]{Department of Mathematical Sciences, University of Memphis, Memphis, TN 38152, USA}
\email{runze.w@hotmail.com}
\thanks{}
\date{\today}
\subjclass[]{}

\begin{document}

\sloppy

\begin{abstract}
   A graph $G=(V,E)$ is called a \emph{$k$-threshold graph} with \emph{thresholds} $\theta_1<\theta_2<...<\theta_k$ if we can assign a real number $r(v)$ to each vertex $v\in V$, such that for any $u,v\in V$, we have $uv\in E$ if and only if $r(u)+r(v)\ge \theta_i$ holds true for an odd number of elements in $\{\theta_1,\theta_2,...,\theta_k\}$. The smallest integer $k$ such that $G$ is a $k$-threshold graph is called the \emph{threshold number} of $G$. For the complete multipartite graphs and the cluster graphs, Kittipassorn and Sumalroj determined the exact threshold numbers of $K_{n\times 3}$ and $nK_3$. In this paper, first we determine the threshold numbers of some path-related graphs, including linear forests, ladders, and tents. Then, on the basis of Kittipassorn and Sumalroj's results, we determine the exact threshold numbers of $K_{n_1\times 1, n_2\times 2, n_3\times 3}$ and $n_1 K_1\cup n_2 K_2\cup n_3 K_3$, which solve a problem proposed by Sumalroj.
\end{abstract}

\maketitle

\section{Introduction}

In this paper, every graph is assumed to be a finite simple graph, and we will use the following notations:
\begin{itemize}
    \item $[a,b]:=\{i\in \Z: a\le i\le b\}$ for $a,b\in \Z$ with $a<b$.
    \item $[c]:=\{j\in \N: 1\le j\le c\}$ for $c\in \N$.
    \item $K_n:=$ the complete graph with $n$ vertices.
    \item $K_{m_1,m_2,...,m_k}:=$ the complete $k$-partite graph with part sizes $m_1,m_2,...,m_k$.
    \item $K_{n_1\times 1, n_2\times 2, ..., n_t\times t}:=$ the complete $(n_1+n_2+...+n_t)$-partite graph with $n_1$ parts having size $1$, $n_2$ parts having size $2$, ..., $n_t$ parts having size $t$.
    \item $K_{m_1}\cup K_{m_2}\cup ...\cup K_{m_k}:=$ the cluster graph with cluster sizes $m_1,m_2,...,m_k$.
    \item $n_1 K_1\cup n_2 K_2\cup ...\cup n_t K_t:=$ the cluster graph with $n_1$ clusters having size $1$, $n_2$ clusters having size $2$, ..., $n_t$ clusters having size $t$.
\end{itemize}

Note that for fixed $t$ and $n_1,n_2,...,n_t$, we have that $n_1 K_1\cup n_2 K_2\cup ...\cup n_t K_t$ is the complement of $K_{n_1\times 1, n_2\times 2, ..., n_t\times t}$.

A graph $G=(V,E)$ is called a \emph{threshold graph} if we can assign a real number $r(v)$ to each vertex $v\in V$, such that for any two vertices $u,v\in V$, we have $uv\in E$ if and only if $r(u)+r(v)\ge 0$. The class of threshold graphs has been extensively studied (see \cite{Go,GJ,GT,HMP,MP}) since it was introduced by Chv\'atal and Hammer \cite{CH77} in 1977.

Jamison and Sprague \cite{JS20} generalized threshold graphs and introduced multithreshold graphs. A graph $G=(V,E)$ is called a \emph{$k$-threshold graph} with \emph{thresholds} $\theta_1<\theta_2<...<\theta_k$ if we can assign a real number $r(v)$ to each vertex $v\in V$, such that for any two vertices $u,v\in V$, we have $uv\in E$ if and only if $r(u)+r(v)\ge \theta_i$ holds true for an odd number of elements in $\{\theta_1,\theta_2,...,\theta_k\}$. We call $r(v)$ the \emph{rank} of $v$. And such a rank assignment $r:V\longrightarrow \R$ is called a \emph{$(\theta_1,\theta_2,...,\theta_k)$-representation} of $G$.

Jamison and Sprague \cite{JS20} proved that every graph with order $n$ is a $k$-threshold graph for some $k\le{n\choose 2}$. So for every finite graph $G$, there is a smallest positive integer $\Theta(G)$, such that $G$ is a $\Theta(G)$-threshold graph. We call $\Theta(G)$ the \emph{threshold number} of $G$.

If $G$ is a $k$-threshold graph with a $(\theta_1,\theta_2,...,\theta_k)$-representation, $H$ is an induced subgraph of $G$, then if we let every vertex in $H$ inherit its rank in the $(\theta_1,\theta_2,...,\theta_k)$-representation of $G$, we will get a $(\theta_1,\theta_2,...,\theta_k)$-representation of $H$. So we can make the following observation.

\begin{proposition} \label{prop}
    Let $G$ be a graph and let $H$ be an induced subgraph of $G$, then $\Theta(H)\le \Theta(G)$.
\end{proposition}

Jamison and Sprague \cite{JS20} showed the threshold number of a graph is very close to the threshold number of its complement.

\begin{theorem}[Jamison and Sprague \cite{JS20}]
    Let $G$ be a graph and let $G^c$ be the complement of $G$, then

    (I) $|\Theta(G)-\Theta(G^c)|\le 1$.

    (II) If $\Theta(G)$ is odd, then $\Theta(G^c)=\Theta(G)$ or $\Theta(G^c)=\Theta(G)-1$.

    (III) If $\Theta(G)$ is even, then $\Theta(G^c)=\Theta(G)$ or $\Theta(G^c)=\Theta(G)+1$.
\end{theorem}

On the threshold numbers of specific graphs, Jamison and Sprague \cite{JS20} showed that the threshold numbers of paths and caterpillars are both at most two, where a caterpillar consists of a path and some leaves attached to the vertices in the path; Wang \cite{Wa} determined the threshold numbers of cycles.

In this paper, we will determine the threshold numbers of some variants of paths, including linear forests, ladders, and tents.

Jamison and Sprague \cite{JS20} also proved a general upper bound on the threshold numbers of the complete multipartite graphs.

\begin{theorem}[Jamison and Sprague \cite{JS20}]
    Let $G$ be a complete $k$-partite graph, then
    \begin{align*}
        \Theta(G)\le 2k.
    \end{align*}
\end{theorem}

They put forward the following problem and conjecture.

\begin{problem}[Jamison and Sprague \cite{JS20}] \label{probjs}
    Determine the exact threshold numbers of $K_{m_1,m_2,...,m_k}$.
\end{problem}

\begin{conjecture}[Jamison and Sprague \cite{JS20}] \label{conj}
    For every integer $k\ge 1$, there is a graph $G$ with $\Theta(G)=2k$ and $\Theta(G^c)=2k+1$.
\end{conjecture}

Chen and Hao \cite{CH22} gave a partial solution of Problem \ref{probjs} by determining the threshold numbers of the complete multipartite graphs with each part not being small, and the threshold numbers of their complements --- the cluster graphs with each cluster not being small.

\begin{theorem}[Chen and Hao \cite{CH22}] 
Let $m_1,m_2,...,m_k$ be $k$ positive integers, if $m_i\ge k+1$ for any $i\in [k]$, then
\begin{align*}
    \Theta(K_{m_1,m_2,...,m_k})=2k-2,
\end{align*}
and
\begin{align*}
    \Theta(K_{m_1}\cup K_{m_2}\cup ...\cup K_{m_k})=2k-1.
\end{align*}
\end{theorem}

This theorem also confirmed Conjecture \ref{conj}.

For the complete multipartite graphs with small parts and the cluster graphs with small clusters, people have studied those graphs with equal-size parts/clusters. Puleo \cite{Pu} proved a lower bound on the threshold numbers of $nK_3$.

\begin{theorem}[Puleo \cite{Pu}]
    $\Theta(nK_3)\ge n^{1/3}$.
\end{theorem}

Based on Puleo's result, Kittipassorn and Sumalroj \cite{KS} further determined the exact threshold numbers of $K_{n\times 3}$, $nK_3$, $K_{n\times 4}$, and $nK_4$.

\begin{theorem}[Kittipassorn and Sumalroj \cite{KS}] \label{ks3}
Let $p_m=m+{m\choose 3}+2$ and let $q_m=m+{m\choose 3}+1$.
    
(I) For $n\ge 2$,
    \[ \Theta(K_{n\times 3})=\begin{cases} 
          2m & if\ n=p_{m-1}, \\
          2m+1 & if\ p_{m-1}+1\le n\le p_m-1.
       \end{cases}
    \]
    
(II) For $n\ge 1$,
    \[ \Theta(nK_3)=\begin{cases} 
          2m-1 & if\ n=q_{m-1}, \\
          2m & if\ q_{m-1}+1\le n\le q_m-1.
       \end{cases}
    \]

\end{theorem}

\begin{theorem}[Kittipassorn and Sumalroj \cite{KS}] \label{ks4}
Let $s_m=m+{\lfloor m/2\rfloor \choose 3}+{\lceil m/2\rceil \choose 3}+2$ and let $t_m=m+{\lfloor m/2\rfloor \choose 3}+{\lceil m/2\rceil \choose 3}+1$.
    
(I) For $n\ge 2$,
    \[ \Theta(K_{n\times 4})=\begin{cases} 
          2m & if\ n=s_{m-1}, \\
          2m+1 & if\ s_{m-1}+1\le n\le s_m-1.
       \end{cases}
    \]
    
(II) For $n\ge 1$,
    \[ \Theta(nK_4)=\begin{cases} 
          2m-1 & if\ n=t_{m-1}, \\
          2m & if\ t_{m-1}+1\le n\le t_m-1.
       \end{cases}
    \]

\end{theorem}

Sumalroj \cite{Su} proposed the following problem for the next goal.
\begin{problem}[Sumalroj \cite{Su}] \label{probsu}
    Determine the exact threshold numbers of $K_{n_1\times 1, n_2\times 2, n_3\times 3}$ and their complements $n_1 K_1\cup n_2 K_2\cup n_3 K_3$.
\end{problem}

In this paper, we determine $\Theta(K_{n_1\times 1, n_2\times 2, n_3\times 3})$ and $\Theta(n_1 K_1\cup n_2 K_2\cup n_3 K_3)$ to give the solution of Problem \ref{probsu}.

\begin{theorem} \label{main}
    Let $n_1,n_2,n_3$ be nonnegative integers, let $p_m=m+{m\choose 3}+2$, and let $q_m=m+{m\choose 3}+1$.

    (I) For $n_1+n_2+n_3\ge 2$:

    \qquad (i) If $n_3\le 2$, then
    \[ \Theta(K_{n_1\times 1, n_2\times 2, n_3\times 3})=\begin{cases} 
          1 & if\ n_2+n_3\le 1, \\
          2 & if\ n_2+n_3=2, \\
          3 & if\ n_2+n_3\ge 3.
       \end{cases}
    \]

    \qquad (ii) If $n_3\ge 3$, then
    \[ \Theta(K_{n_1\times 1, n_2\times 2, n_3\times 3})=\begin{cases} 
          2m & if\ n_3=p_{m-1}, \\
          2m+1 & if\ p_{m-1}+1\le n_3\le p_m-1.
       \end{cases}
    \]

    (II) For $n_1+n_2+n_3\ge 1$:

    \qquad (i) If $n_3\le 1$, then
    \[ \Theta(n_1 K_1\cup n_2 K_2\cup n_3 K_3)=\begin{cases} 
          0 & if\ n_2+n_3=0, \\
          1 & if\ n_2+n_3=1, \\
          2 & if\ n_2+n_3\ge 2.
       \end{cases}
    \]

    \qquad (ii) If $n_3\ge 2$, then
    \[ \Theta(n_1 K_1\cup n_2 K_2\cup n_3 K_3)=\begin{cases} 
          2m-1 & if\ n_3=q_{m-1}, \\
          2m & if\ q_{m-1}+1\le n_3\le q_m-1.
       \end{cases}
    \]
\end{theorem}

Note that Theorem \ref{ks3}, Theorem \ref{ks4}, and Theorem \ref{main} give more examples confirming Conjecture \ref{conj}.

For Theorem \ref{ks3} and Theorem \ref{ks4}, Kittipassorn and Sumalroj \cite{KS} gave detailed proofs for $nK_3$ and $nK_4$, and omitted the proofs for $K_{n\times 3}$ and $K_{n\times 4}$, because the ideas are exactly the same --- for a graph $G$ and its complement $G^c$, whatever we do to an edge in $G$, we can do the same thing to a nonedge in $G^c$.
Theorem \ref{main} can be seen as a generalization of Theorem \ref{ks3}, so when we prove it, we will also concentrate on part (II), which is about $n_1 K_1\cup n_2 K_2\cup n_3 K_3$, and omit the details about $K_{n_1\times 1, n_2\times 2, n_3\times 3}$.

We organize the rest of our paper as follows:
\begin{itemize}
    \item In Section 2, we determine the threshold numbers of linear forests, ladders, and tents.
    \item In Section 3, we first prove (II).(i) of Theorem \ref{main}; then give a review on Kittipassorn and Sumalroj's methodology for proving Theorem \ref{ks3}; on the basis of Kittipassorn and Sumalroj's results, we prove (II).(ii) of Theorem \ref{main}. The key part of our proof is properly assigning ranks to the vertices in $n_1 K_1\cup n_2 K_2$.
    \item In Section 4, we give some remarks and suggest some future work.
\end{itemize}

\section{Variants of paths}
The results in the following lemma have been mentioned in \cite{JS22}, we include a proof here for completeness.

\begin{lemma} \label{lemma}
    We have
    \begin{align*}
        \Theta(2P_2)=\Theta(P_4)=\Theta(C_4)=2.
    \end{align*}
\end{lemma}

\begin{proof}
    For $2P_2$, we assume it consists of two disjoint paths $v_1 v_2$ and $v_3 v_4$. If $\Theta(2P_2)=1$, then there is a rank assignment $r:V\longrightarrow \R$ and a threshold $\theta$ such that
    \[\begin{cases}
        f(v_1)+f(v_2)\ge \theta, \\
        f(v_3)+f(v_4)\ge \theta, \\
        f(v_1)+f(v_3)< \theta, \\
        f(v_2)+f(v_4)< \theta.
    \end{cases}
    \]
    However, these inequalities imply 
    \begin{align*}
        f(v_1)+f(v_3)+f(v_2)+f(v_4)<2\theta\le f(v_1)+f(v_2)+f(v_3)+f(v_4),
    \end{align*}
    contradiction. So $\Theta(2P_2)\ge 2$. Similarly, we have $\Theta(P_4)\ge 2$ and $\Theta(C_4)\ge 2$.

    If we take $\theta_1=-\frac{1}{2}$ and $\theta_2=\frac{1}{2}$, then the rank assignments in Figure \ref{rankassignments} show that $\Theta(2P_2)=\Theta(P_4)=\Theta(C_4)=2$.

    \begin{figure}[H]  
        \tikzset{every picture/.style={line width=0.75pt}} 

\begin{tikzpicture}[x=0.75pt,y=0.75pt,yscale=-1,xscale=1]

\draw   (312,54) -- (386,54) -- (386,128) -- (312,128) -- cycle ;
\draw    (192,54) -- (266,54) ;
\draw    (192,54) -- (192,128) ;
\draw    (192,128) -- (266,128) ;
\draw    (72,54) -- (146,54) ;
\draw    (72,128) -- (146,128) ;
\draw  [fill={rgb, 255:red, 0; green, 0; blue, 0 }  ,fill opacity=1 ] (66.5,54) .. controls (66.5,50.96) and (68.96,48.5) .. (72,48.5) .. controls (75.04,48.5) and (77.5,50.96) .. (77.5,54) .. controls (77.5,57.04) and (75.04,59.5) .. (72,59.5) .. controls (68.96,59.5) and (66.5,57.04) .. (66.5,54) -- cycle ;
\draw  [fill={rgb, 255:red, 0; green, 0; blue, 0 }  ,fill opacity=1 ] (140.5,54) .. controls (140.5,50.96) and (142.96,48.5) .. (146,48.5) .. controls (149.04,48.5) and (151.5,50.96) .. (151.5,54) .. controls (151.5,57.04) and (149.04,59.5) .. (146,59.5) .. controls (142.96,59.5) and (140.5,57.04) .. (140.5,54) -- cycle ;
\draw  [fill={rgb, 255:red, 0; green, 0; blue, 0 }  ,fill opacity=1 ] (66.5,128) .. controls (66.5,124.96) and (68.96,122.5) .. (72,122.5) .. controls (75.04,122.5) and (77.5,124.96) .. (77.5,128) .. controls (77.5,131.04) and (75.04,133.5) .. (72,133.5) .. controls (68.96,133.5) and (66.5,131.04) .. (66.5,128) -- cycle ;
\draw  [fill={rgb, 255:red, 0; green, 0; blue, 0 }  ,fill opacity=1 ] (140.5,128) .. controls (140.5,124.96) and (142.96,122.5) .. (146,122.5) .. controls (149.04,122.5) and (151.5,124.96) .. (151.5,128) .. controls (151.5,131.04) and (149.04,133.5) .. (146,133.5) .. controls (142.96,133.5) and (140.5,131.04) .. (140.5,128) -- cycle ;
\draw  [fill={rgb, 255:red, 0; green, 0; blue, 0 }  ,fill opacity=1 ] (186.5,54) .. controls (186.5,50.96) and (188.96,48.5) .. (192,48.5) .. controls (195.04,48.5) and (197.5,50.96) .. (197.5,54) .. controls (197.5,57.04) and (195.04,59.5) .. (192,59.5) .. controls (188.96,59.5) and (186.5,57.04) .. (186.5,54) -- cycle ;
\draw  [fill={rgb, 255:red, 0; green, 0; blue, 0 }  ,fill opacity=1 ] (260.5,54) .. controls (260.5,50.96) and (262.96,48.5) .. (266,48.5) .. controls (269.04,48.5) and (271.5,50.96) .. (271.5,54) .. controls (271.5,57.04) and (269.04,59.5) .. (266,59.5) .. controls (262.96,59.5) and (260.5,57.04) .. (260.5,54) -- cycle ;
\draw  [fill={rgb, 255:red, 0; green, 0; blue, 0 }  ,fill opacity=1 ] (186.5,128) .. controls (186.5,124.96) and (188.96,122.5) .. (192,122.5) .. controls (195.04,122.5) and (197.5,124.96) .. (197.5,128) .. controls (197.5,131.04) and (195.04,133.5) .. (192,133.5) .. controls (188.96,133.5) and (186.5,131.04) .. (186.5,128) -- cycle ;
\draw  [fill={rgb, 255:red, 0; green, 0; blue, 0 }  ,fill opacity=1 ] (260.5,128) .. controls (260.5,124.96) and (262.96,122.5) .. (266,122.5) .. controls (269.04,122.5) and (271.5,124.96) .. (271.5,128) .. controls (271.5,131.04) and (269.04,133.5) .. (266,133.5) .. controls (262.96,133.5) and (260.5,131.04) .. (260.5,128) -- cycle ;
\draw  [fill={rgb, 255:red, 0; green, 0; blue, 0 }  ,fill opacity=1 ] (306.5,54) .. controls (306.5,50.96) and (308.96,48.5) .. (312,48.5) .. controls (315.04,48.5) and (317.5,50.96) .. (317.5,54) .. controls (317.5,57.04) and (315.04,59.5) .. (312,59.5) .. controls (308.96,59.5) and (306.5,57.04) .. (306.5,54) -- cycle ;
\draw  [fill={rgb, 255:red, 0; green, 0; blue, 0 }  ,fill opacity=1 ] (380.5,54) .. controls (380.5,50.96) and (382.96,48.5) .. (386,48.5) .. controls (389.04,48.5) and (391.5,50.96) .. (391.5,54) .. controls (391.5,57.04) and (389.04,59.5) .. (386,59.5) .. controls (382.96,59.5) and (380.5,57.04) .. (380.5,54) -- cycle ;
\draw  [fill={rgb, 255:red, 0; green, 0; blue, 0 }  ,fill opacity=1 ] (306.5,128) .. controls (306.5,124.96) and (308.96,122.5) .. (312,122.5) .. controls (315.04,122.5) and (317.5,124.96) .. (317.5,128) .. controls (317.5,131.04) and (315.04,133.5) .. (312,133.5) .. controls (308.96,133.5) and (306.5,131.04) .. (306.5,128) -- cycle ;
\draw  [fill={rgb, 255:red, 0; green, 0; blue, 0 }  ,fill opacity=1 ] (380.5,128) .. controls (380.5,124.96) and (382.96,122.5) .. (386,122.5) .. controls (389.04,122.5) and (391.5,124.96) .. (391.5,128) .. controls (391.5,131.04) and (389.04,133.5) .. (386,133.5) .. controls (382.96,133.5) and (380.5,131.04) .. (380.5,128) -- cycle ;

\draw (60,33) node [anchor=north west][inner sep=0.75pt]   [align=left] {$\displaystyle -1$};
\draw (139,33) node [anchor=north west][inner sep=0.75pt]   [align=left] {$\displaystyle 1$};
\draw (65,136) node [anchor=north west][inner sep=0.75pt]   [align=left] {$\displaystyle 0$};
\draw (139,136) node [anchor=north west][inner sep=0.75pt]   [align=left] {$\displaystyle 0$};
\draw (253,13) node [anchor=north west][inner sep=0.75pt]   [align=left] {$\displaystyle -\frac{1}{4}$};
\draw (184,13) node [anchor=north west][inner sep=0.75pt]   [align=left] {$\displaystyle \frac{1}{2}$};
\draw (178,136) node [anchor=north west][inner sep=0.75pt]   [align=left] {$\displaystyle -\frac{3}{4}$};
\draw (260,136) node [anchor=north west][inner sep=0.75pt]   [align=left] {$\displaystyle 1$};
\draw (306,33) node [anchor=north west][inner sep=0.75pt]   [align=left] {$\displaystyle 1$};
\draw (380,136) node [anchor=north west][inner sep=0.75pt]   [align=left] {$\displaystyle 1$};
\draw (299,136) node [anchor=north west][inner sep=0.75pt]   [align=left] {$\displaystyle -1$};
\draw (373,33) node [anchor=north west][inner sep=0.75pt]   [align=left] {$\displaystyle -1$};

\end{tikzpicture}
\caption{Rank assignments for $2P_2$, $P_4$, and $C_4$.}
\label{rankassignments}
    \end{figure}
\end{proof}

Jamison and Sprague \cite{JS20} showed that the threshold number of a path $P_n$ is at most two. (In fact, $\Theta(P_n)=2$ if and only if $n\ge 4$.) Moreover, they showed that the threshold number of a caterpillar is also at most two.

Here we study some other variants of paths.

\subsection{Linear forests}
Taking the disjoint union of multiple paths, we get a \emph{linear forest}. A linear forest consisting of $a_i$ paths of length $i$ for $i\in[n]$ can be denoted by
\begin{align*}
    \bigcup_{i=1}^n a_n P_n.
\end{align*}

\begin{theorem}
    If $\sum_{i=2}^n a_n\ge 2$, then 
    \begin{align*}
        \Theta(\bigcup_{i=1}^n a_n P_n)=2.
    \end{align*}
\end{theorem}

\begin{proof}
    Assuming $\sum_{i=2}^n a_n\ge 2$, we know $\Theta(\bigcup_{i=1}^n a_n P_n)\ge 2$ by Proposition \ref{prop}, the fact that $2P_2$ is an induced subgraph of $\bigcup_{i=1}^n a_n P_n$, and our conclusion that $\Theta(2P_2)=2$.

    Then we show that $\Theta(\bigcup_{i=1}^n a_n P_n)\le 2$. We take the two thresholds to be $\theta_1=-\frac{3}{2}$ and $\theta_2=\frac{3}{2}$. For the rank assignment, first we randomly pick a path $P_{k_1}$ from the $\sum_{i=1}^n a_n$ paths, and assign ranks $1,\ -2,\ 3,\ -4,\ ...,\ (-1)^{k_1-1}k_1$ to its vertices from one end to the other. Then we pick the second path $P_{k_2}$, and assign ranks $k_1+2,\ -(k_1+3),\ k_1+4,\ -(k_1+5),\ ...,\ (-1)^{k_2-1}(k_1+k_2+1)$ to its vertices from one end to the other. In general, for the $i$-th picked path, the rank of the starting endpoint will be 
    \begin{align*}
        |the\ rank\ of\ the\ ending\ endpoint\ of\ the\ (i-1)-th\ path|+2,
    \end{align*}
    and as we move to the ending endpoint, every time we add one and take $+/-$ alternatively. 

    Then it is easy to check that every edge rank sum is $-1$ or $1$, and every nonedge rank sum is smaller than $-\frac{3}{2}$ or greater than $\frac{3}{2}$.
\end{proof}

\subsection{Ladders} A \emph{ladder} of length $n$ is just $P_n\square P_2$, which can be constructed by connecting each pair of vertices at the same position in two $P_n$'s.

\begin{theorem}
    If $n\ge 2$, then 
    \begin{align*}
        \Theta(P_n\square P_2)=2.
    \end{align*}
\end{theorem}

\begin{proof}
    Assuming $n\ge 2$, we have $\Theta(P_n\square P_2)\ge 2$ because $C_4$ is an induced subgraph of $P_n\square P_2$. And we can see $\Theta(P_n\square P_2)\le 2$ by taking $\theta_1=-\frac{3}{2}$ and $\theta_2=\frac{3}{2}$, and using the rank assignment showed in Figure \ref{ladder}.
    \begin{figure}[H]  
        \input ladder.tex
    \end{figure}
\end{proof}

\subsection{Tents}
A \emph{tent} with $n+1$ vertices, denoted by $T_n$, is constructed by connecting a vertex to every vertex on the path $P_n$.

It is easy to see that $\Theta(T_2)=\Theta(T_3)=1$.

\begin{theorem}
    If $n\ge 4$, then
    \begin{align*}
        \Theta(T_n)=3.
    \end{align*}
\end{theorem}

\begin{proof}
    Firstly, if we take $\theta_1=-\frac{3}{2}$, $\theta_2=\frac{3}{2}$, and $\theta_3=5n$, then the rank assignment in Figure \ref{tent} shows that $\Theta(T_n)\le 3$.

    \begin{figure}[H]  
        \tikzset{every picture/.style={line width=0.75pt}} 

\begin{tikzpicture}[x=0.75pt,y=0.75pt,yscale=-1,xscale=1]

\draw  [color={rgb, 255:red, 0; green, 0; blue, 0 }  ,draw opacity=1 ][fill={rgb, 255:red, 0; green, 0; blue, 0 }  ,fill opacity=1 ] (112.5,134) .. controls (112.5,130.96) and (114.96,128.5) .. (118,128.5) .. controls (121.04,128.5) and (123.5,130.96) .. (123.5,134) .. controls (123.5,137.04) and (121.04,139.5) .. (118,139.5) .. controls (114.96,139.5) and (112.5,137.04) .. (112.5,134) -- cycle ;
\draw   (294,38) -- (470,134) -- (118,134) -- cycle ;
\draw    (294,38) -- (174,134) ;
\draw  [color={rgb, 255:red, 0; green, 0; blue, 0 }  ,draw opacity=1 ][fill={rgb, 255:red, 0; green, 0; blue, 0 }  ,fill opacity=1 ] (168.5,134) .. controls (168.5,130.96) and (170.96,128.5) .. (174,128.5) .. controls (177.04,128.5) and (179.5,130.96) .. (179.5,134) .. controls (179.5,137.04) and (177.04,139.5) .. (174,139.5) .. controls (170.96,139.5) and (168.5,137.04) .. (168.5,134) -- cycle ;
\draw    (294,38) -- (225,134) ;
\draw    (294,38) -- (268,134) ;
\draw    (294,38) -- (310,134) ;
\draw  [color={rgb, 255:red, 0; green, 0; blue, 0 }  ,draw opacity=1 ][fill={rgb, 255:red, 0; green, 0; blue, 0 }  ,fill opacity=1 ] (288.5,38) .. controls (288.5,34.96) and (290.96,32.5) .. (294,32.5) .. controls (297.04,32.5) and (299.5,34.96) .. (299.5,38) .. controls (299.5,41.04) and (297.04,43.5) .. (294,43.5) .. controls (290.96,43.5) and (288.5,41.04) .. (288.5,38) -- cycle ;
\draw  [color={rgb, 255:red, 0; green, 0; blue, 0 }  ,draw opacity=1 ][fill={rgb, 255:red, 0; green, 0; blue, 0 }  ,fill opacity=1 ] (219.5,134) .. controls (219.5,130.96) and (221.96,128.5) .. (225,128.5) .. controls (228.04,128.5) and (230.5,130.96) .. (230.5,134) .. controls (230.5,137.04) and (228.04,139.5) .. (225,139.5) .. controls (221.96,139.5) and (219.5,137.04) .. (219.5,134) -- cycle ;
\draw  [color={rgb, 255:red, 0; green, 0; blue, 0 }  ,draw opacity=1 ][fill={rgb, 255:red, 0; green, 0; blue, 0 }  ,fill opacity=1 ] (262.5,134) .. controls (262.5,130.96) and (264.96,128.5) .. (268,128.5) .. controls (271.04,128.5) and (273.5,130.96) .. (273.5,134) .. controls (273.5,137.04) and (271.04,139.5) .. (268,139.5) .. controls (264.96,139.5) and (262.5,137.04) .. (262.5,134) -- cycle ;
\draw  [color={rgb, 255:red, 0; green, 0; blue, 0 }  ,draw opacity=1 ][fill={rgb, 255:red, 0; green, 0; blue, 0 }  ,fill opacity=1 ] (304.5,134) .. controls (304.5,130.96) and (306.96,128.5) .. (310,128.5) .. controls (313.04,128.5) and (315.5,130.96) .. (315.5,134) .. controls (315.5,137.04) and (313.04,139.5) .. (310,139.5) .. controls (306.96,139.5) and (304.5,137.04) .. (304.5,134) -- cycle ;
\draw    (294,38) -- (356,134) ;
\draw    (294,38) -- (411,134) ;
\draw  [color={rgb, 255:red, 0; green, 0; blue, 0 }  ,draw opacity=1 ][fill={rgb, 255:red, 0; green, 0; blue, 0 }  ,fill opacity=1 ] (350.5,134) .. controls (350.5,130.96) and (352.96,128.5) .. (356,128.5) .. controls (359.04,128.5) and (361.5,130.96) .. (361.5,134) .. controls (361.5,137.04) and (359.04,139.5) .. (356,139.5) .. controls (352.96,139.5) and (350.5,137.04) .. (350.5,134) -- cycle ;
\draw  [color={rgb, 255:red, 0; green, 0; blue, 0 }  ,draw opacity=1 ][fill={rgb, 255:red, 0; green, 0; blue, 0 }  ,fill opacity=1 ] (405.5,134) .. controls (405.5,130.96) and (407.96,128.5) .. (411,128.5) .. controls (414.04,128.5) and (416.5,130.96) .. (416.5,134) .. controls (416.5,137.04) and (414.04,139.5) .. (411,139.5) .. controls (407.96,139.5) and (405.5,137.04) .. (405.5,134) -- cycle ;
\draw  [color={rgb, 255:red, 0; green, 0; blue, 0 }  ,draw opacity=1 ][fill={rgb, 255:red, 0; green, 0; blue, 0 }  ,fill opacity=1 ] (464.5,134) .. controls (464.5,130.96) and (466.96,128.5) .. (470,128.5) .. controls (473.04,128.5) and (475.5,130.96) .. (475.5,134) .. controls (475.5,137.04) and (473.04,139.5) .. (470,139.5) .. controls (466.96,139.5) and (464.5,137.04) .. (464.5,134) -- cycle ;

\draw (110,142) node [anchor=north west][inner sep=0.75pt]  [color={rgb, 255:red, 0; green, 0; blue, 0 }  ,opacity=1 ] [align=left] {$\displaystyle 1$};
\draw (157,142) node [anchor=north west][inner sep=0.75pt]  [color={rgb, 255:red, 0; green, 0; blue, 0 }  ,opacity=1 ] [align=left] {$\displaystyle -2$};
\draw (217,142) node [anchor=north west][inner sep=0.75pt]  [color={rgb, 255:red, 0; green, 0; blue, 0 }  ,opacity=1 ] [align=left] {$\displaystyle 3$};
\draw (255,142) node [anchor=north west][inner sep=0.75pt]  [color={rgb, 255:red, 0; green, 0; blue, 0 }  ,opacity=1 ] [align=left] {$\displaystyle -4$};
\draw (356,142) node [anchor=north west][inner sep=0.75pt]  [color={rgb, 255:red, 0; green, 0; blue, 0 }  ,opacity=1 ] [align=left] {...};
\draw (448,140) node [anchor=north west][inner sep=0.75pt]  [color={rgb, 255:red, 0; green, 0; blue, 0 }  ,opacity=1 ] [align=left] {$\displaystyle ( -1)^{n-1} n$};
\draw (279,17) node [anchor=north west][inner sep=0.75pt]  [color={rgb, 255:red, 0; green, 0; blue, 0 }  ,opacity=1 ] [align=left] {$\displaystyle 10n$};

\end{tikzpicture}
\caption{Rank assignment for $T_n$.}
\label{tent}
    \end{figure}

    If $n\ge 4$, then $P_4$ is an induced subgraph of $T_n$, so $\Theta(T_n)\ge 2$. Then, to show that $\Theta(T_n)=3$, we only need to eliminate the possibility that $\Theta(T_n)=2$. Also, we know that $T_4$ is an induced subgraph of $T_n$ for any $n\ge 4$, so it suffices to show that $\Theta(T_4)\neq 2$.

    Assume $f$ is a rank assignment such that the five vertices in $T_4$ receive ranks $a$, $b_1$, $b_2$, $b_3$, and $b_4$ as showed in Figure \ref{T4}, and $\Theta(T_4)=2$ with thresholds $\theta_1<\theta_2$.
    \begin{figure}[H]  
        \tikzset{every picture/.style={line width=0.75pt}} 

\begin{tikzpicture}[x=0.75pt,y=0.75pt,yscale=-1,xscale=1]

\draw  [color={rgb, 255:red, 0; green, 0; blue, 0 }  ,draw opacity=1 ][fill={rgb, 255:red, 0; green, 0; blue, 0 }  ,fill opacity=1 ] (120.5,347) .. controls (120.5,343.96) and (122.96,341.5) .. (126,341.5) .. controls (129.04,341.5) and (131.5,343.96) .. (131.5,347) .. controls (131.5,350.04) and (129.04,352.5) .. (126,352.5) .. controls (122.96,352.5) and (120.5,350.04) .. (120.5,347) -- cycle ;
\draw   (209.5,251) -- (293,347) -- (126,347) -- cycle ;
\draw    (209.5,251) -- (182,347) ;
\draw  [color={rgb, 255:red, 0; green, 0; blue, 0 }  ,draw opacity=1 ][fill={rgb, 255:red, 0; green, 0; blue, 0 }  ,fill opacity=1 ] (176.5,347) .. controls (176.5,343.96) and (178.96,341.5) .. (182,341.5) .. controls (185.04,341.5) and (187.5,343.96) .. (187.5,347) .. controls (187.5,350.04) and (185.04,352.5) .. (182,352.5) .. controls (178.96,352.5) and (176.5,350.04) .. (176.5,347) -- cycle ;
\draw    (209.5,251) -- (233,347) ;
\draw  [color={rgb, 255:red, 0; green, 0; blue, 0 }  ,draw opacity=1 ][fill={rgb, 255:red, 0; green, 0; blue, 0 }  ,fill opacity=1 ] (204,251) .. controls (204,247.96) and (206.46,245.5) .. (209.5,245.5) .. controls (212.54,245.5) and (215,247.96) .. (215,251) .. controls (215,254.04) and (212.54,256.5) .. (209.5,256.5) .. controls (206.46,256.5) and (204,254.04) .. (204,251) -- cycle ;
\draw  [color={rgb, 255:red, 0; green, 0; blue, 0 }  ,draw opacity=1 ][fill={rgb, 255:red, 0; green, 0; blue, 0 }  ,fill opacity=1 ] (227.5,347) .. controls (227.5,343.96) and (229.96,341.5) .. (233,341.5) .. controls (236.04,341.5) and (238.5,343.96) .. (238.5,347) .. controls (238.5,350.04) and (236.04,352.5) .. (233,352.5) .. controls (229.96,352.5) and (227.5,350.04) .. (227.5,347) -- cycle ;
\draw  [color={rgb, 255:red, 0; green, 0; blue, 0 }  ,draw opacity=1 ][fill={rgb, 255:red, 0; green, 0; blue, 0 }  ,fill opacity=1 ] (287.5,347) .. controls (287.5,343.96) and (289.96,341.5) .. (293,341.5) .. controls (296.04,341.5) and (298.5,343.96) .. (298.5,347) .. controls (298.5,350.04) and (296.04,352.5) .. (293,352.5) .. controls (289.96,352.5) and (287.5,350.04) .. (287.5,347) -- cycle ;

\draw (118,355) node [anchor=north west][inner sep=0.75pt]  [color={rgb, 255:red, 0; green, 0; blue, 0 }  ,opacity=1 ] [align=left] {$\displaystyle b_{1}$};
\draw (165,355) node [anchor=north west][inner sep=0.75pt]  [color={rgb, 255:red, 0; green, 0; blue, 0 }  ,opacity=1 ] [align=left] {$\displaystyle b_{2}$};
\draw (225,355) node [anchor=north west][inner sep=0.75pt]  [color={rgb, 255:red, 0; green, 0; blue, 0 }  ,opacity=1 ] [align=left] {$\displaystyle b_{3}$};
\draw (285,355) node [anchor=north west][inner sep=0.75pt]  [color={rgb, 255:red, 0; green, 0; blue, 0 }  ,opacity=1 ] [align=left] {$\displaystyle b_{4}$};
\draw (204,230) node [anchor=north west][inner sep=0.75pt]  [color={rgb, 255:red, 0; green, 0; blue, 0 }  ,opacity=1 ] [align=left] {$\displaystyle a$};

\end{tikzpicture}
\caption{Rank assignment for $T_4$.}
\label{T4}
    \end{figure}

    By definition, we know that $a+b_1$, $a+b_2$, $a+b_3$, $a+b_4$, $b_1+b_2$, $b_2+b_3$, and $b_3+b_4$ are all between $\theta_1$ and $\theta_2$; we also know that either $b_1+b_3\ge \theta_2$, or $b_1+b_3<\theta_1$.

    If $b_1+b_3\ge \theta_2$, then by $b_1+b_2<\theta_2\le b_1+b_3$, we know that $b_2<b_3$; and by $a+b_3<\theta_2\le b_1+b_3$, we know that $a<b_1$. Then, if $b_2+b_4\ge \theta_2$, then by  $b_3+b_4<\theta_2\le b_2+b_4$, we have that $b_3<b_2$, contradiction. So we must have $b_2+b_4<\theta_1$. By $b_2+b_4<\theta_1\le a+b_2$, we know $b_4<a$. Now, if $b_1+b_4\ge \theta_2$, then we will have $a+b_1<\theta_2\le b_1+b_4$, so $a<b_4$, contradiction; and if $b_1+b_4<\theta_1$, then we will have $b_1+b_4<\theta_1\le a+b_4$, so $b_1<a$, contradiction. Thus, we conclude that $b_1+b_3\ge \theta_2$ is not possible.

    Similarly, we can prove that $b_1+b_3<\theta_1$ is not possible either. So $\Theta(T_4)\neq 2$, and this is sufficient to finish the proof.
\end{proof}

\section{Cluster graphs with small clusters}

\subsection{When $n_3$ is very small}
We prove (II).(i) of Theorem \ref{main} by making usage of Proposition \ref{prop} and the following lemma, which is just a special case of Lemma \ref{lemma}, because $K_2$ is the same as $P_2$. The proof of (I).(i) is very similar, so we omit it.

\begin{lemma} \label{lemmanK2}
For the cluster graphs with each cluster having size $2$,
    \[ \Theta(n K_2)=\begin{cases} 
          1 & if\ n=1, \\
          2 & if\ n\ge 2.
       \end{cases}
    \]
\end{lemma}

Now let us prove (II).(i) of Theorem \ref{main}.

\begin{proof}[Proof of (II).(i)]
    We prove: For $n_1+n_2+n_3\ge 1$, if $n_3\le 1$, then
    \[ \Theta(n_1 K_1\cup n_2 K_2\cup n_3 K_3)=\begin{cases} 
          0 & if\ n_2+n_3=0, \\
          1 & if\ n_2+n_3=1, \\
          2 & if\ n_2+n_3\ge 2.
       \end{cases}
    \]
    
    Firstly, if $n_2+n_3=0$, then the graph has no edge, only $n_1$ independent vertices, so $\Theta(n_1 K_1\cup n_2 K_2\cup n_3 K_3)=0$.

    Secondly, if $n_2+n_3=1$, then $K_2$ is an induced subgraph of $n_1 K_1\cup n_2 K_2\cup n_3 K_3$, so by Proposition \ref{prop} and Lemma \ref{lemmanK2}, we know $\Theta(n_1 K_1\cup n_2 K_2\cup n_3 K_3)\ge 1$. Now we let $\theta_1=0$, let all vertices in that single $K_2$-cluster or $K_3$-cluster have rank $1$, and let all vertices in $n_1 K_1$ have rank $-2$. Then every edge rank sum is $1+1=2>0$, and every nonedge rank sum is $1-2=-1$ or $-2-2=-4$, both smaller than $0$. So we have a $(0)$-representation of $n_1 K_1\cup n_2 K_2\cup n_3 K_3$. And thus $\Theta(n_1 K_1\cup n_2 K_2\cup n_3 K_3)=1$.

    Lastly, if $n_2+n_3\ge 2$, then $(n_2+n_3)K_2$ is an induced subgraph of $n_1 K_1\cup n_2 K_2\cup n_3 K_3$, so by Proposition \ref{prop} and Lemma \ref{lemmanK2}, we know $\Theta(n_1 K_1\cup n_2 K_2\cup n_3 K_3)\ge 2$. Now if $n_3=0$, we choose $\theta_1=-\frac{1}{2}$ and $\theta_2=\frac{1}{2}$. Denote the set of $K_1$-clusters by $\{K_1^1, K_1^2, ..., K_1^{n_1}\}$, with each $K_1^i$ consisting of one single vertex $u_i$. Denote the set of $K_2$-clusters by $\{K_2^1, K_2^2, ..., K_2^{n_2}\}$, with each $K_2^j$ consisting of two vertices $v_j$, $w_j$, and an edge between them. Let $r(u_i)=i$ for any $i\in[n_1]$. Let $r(v_j)=n_1+j$ and $r(w_j)=-n_1-j$ for any $j\in [n_2]$. Then every edge rank sum is $n_1+j-n_1-j=0$, and every nonedge rank sum is $i_1+i_2>\frac{1}{2}$ for $i_1,i_2\in [n_1]$, or $i+n_1+j>\frac{1}{2}$ for $i\in [n_1]$ and $j\in [n_2]$, or $i-n_1-j<-\frac{1}{2}$ for $i\in [n_1]$ and $j\in [n_2]$. So we have a $(-\frac{1}{2},\frac{1}{2})$-representation of $n_1 K_1\cup n_2 K_2\cup n_3 K_3$, and thus $\Theta(n_1 K_1\cup n_2 K_2\cup n_3 K_3)=2$. If $n_3=1$, then we choose the same thresholds and assign the same ranks to the vertices in $n_1 K_1\cup n_2 K_2$ as what we did in the case $n_3=0$, and let all three vertices in the $K_3$-cluster have rank 0. Now it is easy to check that every edge rank sum is 0, and every nonedge rank sum is either greater than $\frac{1}{2}$ or smaller than $-\frac{1}{2}$, so again we have a $(-\frac{1}{2},\frac{1}{2})$-representation of $n_1 K_1\cup n_2 K_2\cup n_3 K_3$, and thus $\Theta(n_1 K_1\cup n_2 K_2\cup n_3 K_3)=2$.
\end{proof}

\subsection{A review}

We will briefly review Kittipassorn and Sumalroj's methodology for proving Theorem \ref{ks3}. On the basis of it, we get $n_1 K_1\cup n_2 K_2$ involved and prove (II).(ii). The proof of (I).(ii) is very similar.

Let us begin the review. Firstly, to give a lower bound on $\Theta(n_3 K_3)$, we assign a color to each edge in $n_3 K_3$ in the following fashion:

In a $(\theta_1, \theta_2,..., \theta_k)$-representation of $n_3 K_3$ where $\theta_1<\theta_2<...<\theta_k$, an edge $xy$ is colored with color $i$, for $i\in\{1,2,..., \lceil \frac{k}{2}\rceil\}$, if $r(x)+r(y)\in[\theta_{2i-1},\theta_{2i})$ or $r(x)+r(y)\in[\theta_{2\lceil \frac{k}{2}\rceil-1},\infty)$ in the case $i=\lceil \frac{k}{2}\rceil$ and $k$ is odd. A $K_3$-cluster is called $ij\ell$ if its edges are colored by $i$, $j$, and $\ell$.

There are two results indicating that two $K_3$-clusters cannot have too much similarity.

\begin{lemma}[Puleo \cite{Pu}] \label{ijl}
    In a $(\theta_1, \theta_2,..., \theta_k)$-representation of $n_3 K_3$, for any colors $i, j, \ell\in \{1,2,..., \lceil \frac{k}{2}\rceil\}$, $ij\ell$ cannot appear more than once.
\end{lemma}
In other words, two $K_3$-clusters cannot be colored in the same way.
\begin{lemma}[Kittipassorn and Sumalroj \cite{KS}] \label{ijjill}
    In a $(\theta_1, \theta_2,..., \theta_k)$-representation of $n_3 K_3$, for any colors $i,j,\ell\in \{1,2,..., \lceil \frac{k}{2}\rceil\}$, $ijj$ and $i\ell\ell$ cannot simultaneously appear.
\end{lemma}

Let $q_m=m+{m\choose 3}+1$. For $q_{m-1}\le n_3\le q_m-1$, it is easy to check that we need at least $m$ colors to color every edge in $n_3 K_3$, to make sure that we have no conflict with Lemma \ref{ijl} and Lemma \ref{ijjill}. And thus we need at least $2m-1$ thresholds to provide enough colors.

\begin{lemma}[Kittipassorn and Sumalroj \cite{KS}]
    $\Theta(n_3 K_3)\ge 2m-1$. 
\end{lemma}

To get an upper bound:

A rank assignment of $n_3 K_3$ is called an \emph{$\{a_1,a_2,...,a_m\}$-assignment} if every $K_3$-cluster has edge rank sums in the form $a_i,a_i,a_i$ for $i\in [m]$ or $a_i,a_j,a_\ell$ for distinct $i,j,\ell\in[m]$, and no two $K_3$-clusters have the same multiset of edge rank sums. A $K_3$-cluster is called $K_3(a_i,a_j,a_\ell)$ if its edge rank sums are $a_i$, $a_j$, and $a_\ell$. Note that a $K_3$-cluster is $K_3(a_i,a_i,a_i)$ if and only if all its three vertices have rank $\frac{a_i}{2}$; and a $K_3$-cluster is $K_3(a_i, a_j, a_\ell)$ if and only if its three vertices have ranks $\frac{a_i+a_j-a_\ell}{2}$, $\frac{a_i+a_\ell-a_j}{2}$, and $\frac{a_j+a_\ell-a_i}{2}$.

For $q_{m-1}\le n_3\le q_m-1$, indeed we have an $\{a_1,a_2,...,a_m\}$-assignment of $n_3 K_3$, because we have $m$ available $K_3$-clusters in the form $K_3(a_i,a_i,a_i)$ and ${m\choose 3}$ available $K_3$-clusters in the form $K_3(a_i, a_j, a_\ell)$, and $n_3\le q_m-1=m+{m\choose 3}$.

Kittipassorn and Sumalroj showed that we can properly choose values for $a_1, a_2, ..., a_m$, such that in an $\{a_1, a_2, ..., a_m\}$-assignment of $n_3 K_3$, every nonedge rank sum is not equal to any edge rank sum.

\begin{lemma}[Kittipassorn and Sumalroj \cite{KS}]
    Let $\{a_1, a_2, ..., a_m\}\subset \R^+$ be a linearly independent set over $\Q$. In an $\{a_1, a_2, ..., a_m\}$-assignment of $n_3 K_3$, the edge rank sums and the nonedge rank sums do not coincide.
\end{lemma}

For linear independence, for example, we can choose $a_1, a_2, ..., a_m$ to be the square roots of the first $m$ prime numbers.

We let such $a_1, a_2, ..., a_m$ with $a_1<a_2<...<a_m$ be the edge rank sums, and now we have the set of edge rank sums being disjoint from the set of nonedge rank sums. This means the distance between any edge rank sum and any nonedge rank sum is positive. Also we know the distance between any two edge rank sums is positive, as they have been chosen to be distinct numbers. Note that two nonedge rank sums could be the same, they could have distance 0, but this does not affect our following steps.

Let $\Delta$ be the smallest distance between an edge rank sum and another rank sum (edge or nonedge). We have showed $\Delta>0$. Then we choose $0<\epsilon<\Delta$, and let $a_i$ and $a_i+\epsilon$ be the thresholds for every $i\in [m]$. Now every edge rank sum $a_i$ is greater than or equal to an odd number of thresholds, which are $a_1, a_1+\epsilon, a_2, a_2+\epsilon, ..., a_i$. Every nonedge rank sum is greater than or equal to an even number of thresholds, which are $a_1, a_1+\epsilon, a_2, a_2+\epsilon, ..., a_j, a_j+\epsilon$ for some $j\in [m]$. So we have an $(a_1, a_1+\epsilon, a_2, a_2+\epsilon, ..., a_m, a_m+\epsilon)$-representation of $n_3 K_3$. Thus $2m$ thresholds are enough.

\begin{lemma}[Kittipassorn and Sumalroj \cite{KS}]
    $\Theta(n_3 K_3)\le 2m$.
\end{lemma}

Let us recall their conclusions here.

    \[ \Theta(n_3 K_3)=\begin{cases} 
          2m-1 & if\ n_3=q_{m-1}, \\
          2m & if\ q_{m-1}+1\le n_3\le q_m-1.
       \end{cases}
    \]

For $q_{m-1}\le n_3\le q_m-1$, we already know the bounds $2m-1\le \Theta(n_3 K_3)\le 2m$. The exact threshold numbers of $n_3 K_3$ would be determined by using the following two results.

\begin{lemma}[Kittipassorn and Sumalroj \cite{KS}] \label{lemmaks1}
    If $n_3=q_{m-1}$, then $\Theta(n_3 K_3)\le 2m-1$.
\end{lemma}

\begin{lemma}[Kittipassorn and Sumalroj \cite{KS}] \label{lemmaks2}
    If $q_{m-1}+1\le n_3\le q_m-1$, then $\Theta(n_3 K_3)\ge 2m$.
\end{lemma}

They proved Lemma \ref{lemmaks1} by choosing $\{a_1, a_2, ..., a_m\}\subset \R^+$ to be a linearly independent set over $\Q$ with $a_1<a_2<...<a_{m-1}\le \frac{a_m}{2}$, and letting the $q_{m-1}=(m-1)+{m-1\choose 3}+1$ $K_3$-clusters be $K_3(a_i,a_i,a_i)$ with $i\in [m]$ and $K_3(a_i, a_j, a_\ell)$ with distinct $i, j, \ell\in [m-1]$. In this way, no rank sum exceeds $a_m$, so the largest threshold $a_m+\epsilon$ is not needed, $2m-1$ thresholds are enough. The specific details of Lemma \ref{lemmaks2} will not be used in this paper, readers may refer to \cite{KS} for more information.

Now we have finished reviewing the structure of the proof of Theorem \ref{ks3}.

\subsection{When $n_3$ is larger}

We are going to get $n_1 K_1\cup n_2 K_2$ involved, and prove (II).(ii). The proof of (I).(ii) is very similar, so it is omitted here.

Let $q_m=m+{m\choose 3}+1$. Firstly, for $q_{m-1}\le n_3\le q_m-1$, we show that we can find an $\{a_1, a_2, ..., a_m\}$-assignment of $n_3 K_3$, where we have all $K_3(a_i,a_i,a_i)$'s.

\begin{lemma} \label{lemmafinding}
    If $q_{m-1}\le n_3\le q_m-1$, then for any $\{a_1,a_2,...,a_m\}\subset \R^+$ being a linearly independent set over $\Q$ with $a_1<a_2<...<a_m$, there is an $\{a_1, a_2, ..., a_m\}$-assignment of $n_3 K_3$ such that for every $i\in [m]$, $K_3(a_i,a_i,a_i)$ appears.
\end{lemma}
\begin{proof}
    If $q_{m-1}+1\le n_3\le q_m-1$, then we first follow the idea we just reviewed to find the $\{a_1, a_2, ..., a_m\}$-assignment of $(q_m-1)K_3$. Note that for chosen $a_1, a_2, ..., a_m$, there is only one $\{a_1, a_2, ..., a_m\}$-assignment of $(q_m-1)K_3$, which consists of every $K_3(a_i,a_i,a_i)$ with $i\in[m]$ and every $K_3(a_i,a_j,a_\ell)$ with distinct $i,j,\ell\in [m]$. Then we delete $q_m-1-n_3$ $K_3$-clusters to get an $\{a_1, a_2, ..., a_m\}$-assignment of $n_3 K_3$. We have $q_m-1-n_3\le q_m-1-(q_{m-1}+1)={m\choose 3}-{m-1\choose 3}-1<{m\choose 3}$. So we can randomly delete $q_m-1-n_3$ $K_3$-clusters from those ${m\choose 3}$ $K_3$-clusters with three different edge rank sums, and have all $K_3(a_i,a_i,a_i)$'s reserved.

    If $n_3=q_{m-1}$, the same as in Lemma \ref{lemmaks1}, we choose the $q_{m-1}=(m-1)+{m-1\choose 3}+1$ $K_3$-clusters to be $K_3(a_i,a_i,a_i)$ with $i\in [m]$ and $K_3(a_i, a_j, a_\ell)$ with distinct $i, j, \ell\in [m-1]$, where we have all $K_3(a_i,a_i,a_i)$'s.
\end{proof}

Now we can do the proof of (II).(ii).

\begin{proof}[Proof of (II).(ii)]
We prove: For $n_1+n_2+n_3\ge 1$, if $n_3\ge 2$, then
    \[ \Theta(n_1 K_1\cup n_2 K_2\cup n_3 K_3)=\begin{cases} 
          2m-1 & if\ n_3=q_{m-1}, \\
          2m & if\ q_{m-1}+1\le n_3\le q_m-1,
       \end{cases}
    \]
where $q_m=m+{m\choose 3}+1$.

Firstly, we have $m\ge 2$ because $n_3\ge 2=q_1$. Also, we may assume $n_1+n_2\ge 1$, because the case $n_1+n_2=0$ is proved by Theorem \ref{ks3}.

We know that $n_3 K_3$ is an induced subgraph of $n_1 K_1\cup n_2 K_2\cup n_3 K_3$, so by Proposition \ref{prop} and Theorem \ref{ks3}, we have 

\begin{lemma} \label{lemmalowerbound}
    \[ \Theta(n_1 K_1\cup n_2 K_2\cup n_3 K_3)\ge \begin{cases} 
          2m-1 & if\ n_3=q_{m-1}, \\
          2m & if\ q_{m-1}+1\le n_3\le q_m-1.
       \end{cases}
    \]
\end{lemma}

Using Lemma \ref{lemmafinding}, we find an $\{a_1, a_2, ..., a_m\}$-assignment of $n_3 K_3$ with $a_1<a_2<...<a_m$, where we have $K_3(a_i,a_i,a_i)$ for every $i\in [m]$. Note that if $n_3=q_{m-1}$, then we need an $\{a_1, a_2, ..., a_m\}$-assignment with $a_1<a_2<...<a_{m-1}\le \frac{a_m}{2}$, e.g. we can let $a_1,a_2,...,a_{m-1}$ be the square roots of the first $m-1$ prime numbers, and let $a_m$ be the square root of some sufficiently large prime number. 

Let $\Delta$ be the smallest distance between an edge rank sum in $n_3 K_3$ and another rank sum (edge or nonedge) in $n_3 K_3$, and let $\epsilon=\frac{\Delta}{2(n_1+n_2)}$. We know $\Delta>0$, and hence $\epsilon>0$. By the definitions of $a_1, a_2, ..., a_m$ and $\epsilon$, we have that $a_1<a_1+\epsilon<a_2<a_2+\epsilon<...<a_m<a_m+\epsilon$. Denote $\{a_1,a_1+\epsilon,a_2,a_2+\epsilon,..., a_m,a_m+\epsilon\}$ by $\mathcal{A}$. Denote the set of $K_1$-clusters by $\{K_1^1, K_1^2, ..., K_1^{n_1}\}$ with each $K_1^i$ consisting of a single vertex $u_i$. Let the rank of $u_i$ be $\frac{a_1}{2}+i\epsilon$ for $i\in[n_1]$. Denote the set of $K_2$-clusters by $\{K_2^1, K_2^2, ..., K_2^{n_2}\}$ with each $K_2^j$ consisting of two vertices $v_j$, $w_j$, and an edge between them. Let the rank of $v_j$ be $\frac{a_1}{2}+(n_1+j)\epsilon$ and the rank of $w_j$ be $\frac{a_1}{2}-(n_1+j)\epsilon$ for $j\in[n_2]$.

If $q_{m-1}+1\le n_3\le q_m-1$, we will show that our rank assignment is an $(a_1,a_1+\epsilon,a_2,a_2+\epsilon,..., a_m,a_m+\epsilon)$-representation of $n_1 K_1\cup n_2 K_2\cup n_3 K_3$, which implies $\Theta(n_1 K_1\cup n_2 K_2\cup n_3 K_3)\le 2m$. Combining this result with Lemma \ref{lemmalowerbound}, we will get $\Theta(n_1 K_1\cup n_2 K_2\cup n_3 K_3)=2m$.

Now we need to do some tedious case-by-case calculations in order to check that every edge rank sum is greater than or equal to an odd number of elements in $\mathcal{A}$, and every nonedge rank sum is greater than or equal to an even number of elements in $\mathcal{A}$.

For rank sums involving both vertices in $n_3 K_3$, it is the same as what we have in the proof of Theorem \ref{ks3}: Every edge rank sum $a_i$ is greater than or equal to an odd number of elements in $\mathcal{A}$, which are $a_1, a_1+\epsilon, a_2, a_2+\epsilon, ..., a_i$. And by the definitions of $\Delta$ and $\epsilon$, every nonedge rank sum is greater than or equal to an even number of elements in $\mathcal{A}$, which are $a_1, a_1+\epsilon, a_2, a_2+\epsilon, ..., a_j, a_j+\epsilon$ for some $j\in [m]$.

For rank sums involving at least one vertex in $n_1 K_1\cup n_2 K_2$:

It is clear that every edge with at least one vertex in $n_1 K_1\cup n_2 K_2$ is actually an edge in some $K_2$-cluster, so its edge rank sum is $r(v_j)+r(w_j)=\frac{a_1}{2}+(n_1+j)\epsilon+\frac{a_1}{2}-(n_1+j)\epsilon=a_1$, which means $|\{\alpha\in\mathcal{A}:\alpha\le r(v_j)+r(w_j)\}|=1$, an odd number.

And for a nonedge rank sum involving at least one vertex in $n_1 K_1\cup n_2 K_2$:

\textbf{Case 1.} One vertex is in $n_3 K_3$, the other is in $n_1 K_1\cup n_2 K_2$.

Suppose $x\in n_3 K_3$, $y\in n_1 K_1\cup n_2 K_2$, and $xy$ is a nonedge. 

\textbf{Subcase 1.1.} $x\in K_3(a_1,a_1,a_1)$.

In this subcase, $r(x)=\frac{a_1}{2}$. So $r(x)+r(y)=a_1+i\epsilon$ with $i\in [n_1+n_2]$, or $r(x)+r(y)=a_1-j\epsilon$ with $j\in [n_1+1, n_1+n_2]$.

If $r(x)+r(y)=a_1+i\epsilon$ with $i\in [n_1+n_2]$, then $a_1<a_1+\epsilon\le r(x)+r(y)$, and $r(x)+r(y)\le a_1+(n_1+n_2)\epsilon=a_1+\frac{1}{2}\Delta<a_1+\Delta\le a_2$. So $|\{\alpha\in\mathcal{A}:\alpha\le r(x)+r(y)\}|=2$, an even number.

If $r(x)+r(y)=a_1-j\epsilon$ with $j\in [n_1+1, n_1+n_2]$, then $r(x)+r(y)<a_1$, so $|\{\alpha\in\mathcal{A}:\alpha\le r(x)+r(y)\}|=0$, an even number.

\textbf{Subcase 1.2.} $x\notin K_3(a_1,a_1,a_1)$.

In this subcase, we find another vertex $z\in K_3(a_1,a_1,a_1)$, so $r(z)=\frac{a_1}{2}$. We know $xz$ is a nonedge, so $|\{\alpha\in\mathcal{A}:\alpha\le r(x)+r(z)=r(x)+\frac{a_1}{2}\}|$ is an even number. We know $r(x)$ is either $\frac{a_i+a_j-a_{\ell}}{2}\le \frac{a_m+a_{m-1}-a_1}{2}$ for distinct $i,j,\ell\in [m]$, or $\frac{a_i}{2}\le \frac{a_m}{2}$ for some $i\in[2, m]$. So $r(x)+\frac{a_1}{2}<a_m$. And by the definition of $\Delta$, we know that for any $i\in[m]$, the distance between $r(x)+\frac{a_1}{2}$ (a nonedge rank sum) and $a_i$ (an edge rank sum) is at least $\Delta$. So either $a_s+\Delta\le r(x)+\frac{a_1}{2}\le a_{s+1}-\Delta$ for some $s\in [m-1]$, or $r(x)+\frac{a_1}{2}\le a_1-\Delta$.

\textbf{Subsubcase 1.2.1.} $a_s+\Delta\le r(x)+\frac{a_1}{2}\le a_{s+1}-\Delta$ for some $s\in [m-1]$.

If $r(y)=\frac{a_1}{2}+i\epsilon$ with $i\in[n_1+n_2]$, then $a_1<a_1+\epsilon<...<a_s+\epsilon<a_s+\Delta\le r(x)+\frac{a_1}{2}< r(x)+r(y)$, and $r(x)+r(y)\le r(x)+\frac{a_1}{2}+(n_1+n_2)\epsilon\le a_{s+1}-\Delta+(n_1+n_2)\epsilon=a_{s+1}-\Delta+\frac{1}{2}\Delta<a_{s+1}$. So $|\{\alpha\in\mathcal{A}:\alpha\le r(x)+r(y)\}|=2s$, an even number.

If $r(y)=\frac{a_1}{2}-j\epsilon$ with $j\in [n_1+1, n_1+n_2]$, then $a_1<a_1+\epsilon<...<a_s+\epsilon=a_s+\frac{\Delta}{2(n_1+n_2)}<a_s+\Delta-\frac{1}{2}\Delta=a_s+\Delta-(n_1+n_2)\epsilon\le a_s+\Delta-j\epsilon\le r(x)+\frac{a_1}{2}-j\epsilon=r(x)+r(y)$, and $r(x)+r(y)=r(x)+\frac{a_1}{2}-j\epsilon<r(x)+\frac{a_1}{2}\le a_{s+1}-\Delta<a_{s+1}$. So $|\{\alpha\in\mathcal{A}:\alpha\le r(x)+r(y)\}|=2s$, an even number.

\textbf{Subsubcase 1.2.2.} $r(x)+\frac{a_1}{2}\le a_1-\Delta$.

If $r(y)=\frac{a_1}{2}+i\epsilon$ with $i\in[n_1+n_2]$, then $r(x)+r(y)\le r(x)+\frac{a_1}{2}+(n_1+n_2)\epsilon\le a_1-\Delta+(n_1+n_2)\epsilon=a_1-\Delta+\frac{1}{2}\Delta<a_1$. So $|\{\alpha\in\mathcal{A}:\alpha\le r(x)+r(y)\}|=0$, an even number.

If $r(y)=\frac{a_1}{2}-j\epsilon$ with $j\in [n_1+1, n_1+n_2]$, then $r(x)+r(y)=r(x)+\frac{a_1}{2}-j\epsilon<r(x)+\frac{a_1}{2}\le a_1-\Delta<a_1$. So $|\{\alpha\in\mathcal{A}:\alpha\le r(x)+r(y)\}|=0$, an even number.

\textbf{Case 2.} Both vertices are in $n_1 K_1\cup n_2 K_2$.

Suppose the two vertices are $x$ and $y$, we know they are not in the same $K_2$-cluster because $xy$ is a nonedge. So we can assume $r(x)=\frac{a_1}{2}+i\epsilon$ and $r(y)=\frac{a_1}{2}+j\epsilon$ with $i,j\in [-(n_1+n_2), -(n_1+1)]\cup [1, n_1+n_2]$, $i\neq j$, and $i\neq -j$. So $r(x)+r(y)=a_1+k\epsilon$ with $k\in [-2(n_1+n_2)+1, 2(n_1+n_2)-1]\setminus \{0\}$.

If $k\in [-2(n_1+n_2)+1, -1]$, then $r(x)+r(y)=a_1+k\epsilon\le a_1-\epsilon<a_1$, so $|\{\alpha\in\mathcal{A}:\alpha\le r(x)+r(y)\}|=0$, an even number.

If $k\in [1, 2(n_1+n_2)-1]$, then $a_1<a_1+\epsilon\le 
a_1+k\epsilon=r(x)+r(y)$, and $r(x)+r(y)=a_1+k\epsilon\le a_1+(2(n_1+n_2)-1)\epsilon=a_1+\frac{2(n_1+n_2)-1}{2(n_1+n_2)}\Delta<a_1+\Delta\le a_2$. So $|\{\alpha\in\mathcal{A}:\alpha\le r(x)+r(y)\}|=2$, an even number.

We have completed the proof of the case $q_{m-1}+1\le n_3\le q_m-1$.

If $n_3=q_{m-1}$, we will show that our rank assignment is an $(a_1,a_1+\epsilon,a_2,a_2+\epsilon,..., a_m)$-representation of $n_1 K_1\cup n_2 K_2\cup n_3 K_3$, which implies $\Theta(n_1 K_1\cup n_2 K_2\cup n_3 K_3)\le 2m-1$. Combining this result with Lemma \ref{lemmalowerbound}, we will get $\Theta(n_1 K_1\cup n_2 K_2\cup n_3 K_3)=2m-1$. Note that in the case $n_3=q_{m-1}$, we have chosen $\{a_1,a_2,...,a_m\}\subset \R^+$ to be a linearly independent set over $\Q$ such that $a_1<a_2<...<a_{m-1}\le \frac{a_m}{2}$, and have chosen the $q_{m-1}$ $K_3$-clusters to be $K_3(a_i,a_i,a_i)$ with $i\in [m]$ and $K_3(a_i, a_j, a_\ell)$ with distinct $i, j, \ell\in [m-1]$.

Firstly, we check that no rank sum exceeds $a_m$ (so that we know in this case the threshold $a_m+\epsilon$ is not needed). 
The rank of a vertex in $n_3 K_3$ is either $\frac{a_i}{2}\le \frac{a_m}{2}$ for some $i\in [m]$, or $\frac{a_i+a_j-a_\ell}{2}\le \frac{a_{m-1}+a_{m-2}-a_1}{2}<\frac{a_m}{2}$ for distinct $i, j, \ell\in [m-1]$. So a rank sum involving both vertices in $n_3 K_3$ is at most $a_m$. Furthermore, the rank of a vertex in $n_1 K_1\cup n_2 K_2$ is at most $\frac{a_1}{2}+(n_1+n_2)\epsilon=\frac{a_1}{2}+\frac{1}{2}\Delta\le \frac{a_2}{2}\le \frac{a_m}{2}$, so a rank sum involving at least one vertex in $n_1 K_1\cup n_2 K_2$ is also at most $a_m$. 

Then, we denote $\{a_1,a_1+\epsilon,a_2,a_2+\epsilon,..., a_m\}$ by $\mathcal{A}'$, and the same as what we did in the case $q_{m-1}+1\le n_3\le q_m-1$, we can check that every edge rank sum is greater than or equal to an odd number of elements in $\mathcal{A}'$, and every nonedge rank sum is greater than or equal to an even number of elements in $\mathcal{A}'$.

The case $n_3=q_{m-1}$ is also proved.
\end{proof}

\section{Remarks}

For $K_{m_1}\cup K_{m_2}\cup ...\cup K_{m_k}$ and $K_{m_1,m_2,...,m_k}$, on the basis of Kittipassorn and Sumalroj's work \cite{KS}, we have determined their exact threshold numbers in the case $m_i\le 3$ for any $i\in [k]$; Chen and Hao \cite{CH22} have determined their exact threshold numbers in the case $m_i\ge k+1$ for any $i\in [k]$.

For general cluster graphs $\bigcup_{s=1}^t n_s K_s=n_1 K_1\cup n_2 K_2\cup ...\cup n_t K_t$, we can make the following observations. Similar conclusions for general complete multipartite graphs $K_{n_1\times 1, n_2\times 2, ..., n_t\times t}$ can also be drawn.

\begin{proposition}
    Let $n_1,n_2,...,n_t$ be nonnegative integers with $\sum_{i=1}^t n_i\ge 1$, and let $q_m=m+{m\choose 3}+1$.

    \qquad (I) If $\sum_{j=3}^t n_j\le 1$, then
    \[ \Theta(\bigcup_{s=1}^t n_s K_s)=\begin{cases} 
          0 & if\ \sum_{k=2}^t n_k=0, \\
          1 & if\ \sum_{k=2}^t n_k=1, \\
          2 & if\ \sum_{k=2}^t n_k\ge 2.
       \end{cases}
    \]

    \qquad (II) If $\sum_{j=3}^t n_j\ge 2$, then
    \[ \Theta(\bigcup_{s=1}^t n_s K_s)=\begin{cases} 
          2m-1 & if\ \sum_{j=3}^t n_j=q_{m-1}\ and\ \sum_{\ell=4}^t n_\ell\le m,\\
          2m & if\ q_{m-1}+1\le \sum_{j=3}^t n_j\le q_m-1\ and\ \sum_{\ell=4}^t n_\ell\le m.
       \end{cases}
    \]
\end{proposition}
We can check (I) as what we did in Section 2. And for (II), the basic idea is: If $\sum_{\ell=4}^t n_\ell\le m$, then for each $K_\ell$-cluster with $\ell\in [4,t]$, we can associate it with a distinct $i\in [m]$, and make it a $K_\ell(\underbrace{a_i,a_i,...,a_i}_{\ell})$, which means we let all its edge rank sums be $a_i$. And then these $K_\ell$-clusters will behave similarly to those $K_3(a_i,a_i,a_i)$'s in Section 3 and Section 4.

For the exact threshold numbers of specific graphs, we can consider the following problems.

\begin{problem}
    Determine the exact threshold number of the $m\times n$ grid $P_m\square P_n$.
\end{problem}

\begin{problem}
    Determine the exact threshold numbers of spiders, where a spider is a tree with exactly one vertex having degree $\ge 3$.
\end{problem}

For the cluster graphs (and the complete multipartite graphs), as Sumalroj \cite{Su} also suggested, we may take $\Theta(n_3 K_3\cup n_4 K_4)$ or $\Theta(n K_5)$ to be the next goal.

According to our current results, it is predictable that the exact threshold numbers of general cluster graphs/complete multipartite graphs will not be in a very neat form. So, finding some asymptotic results for $\Theta(\bigcup_{s=1}^t n_s K_s)$ and $\Theta(K_{n_1\times 1, n_2\times 2, ..., n_t\times t})$ could also be a feasible future goal.

\end{document}